\documentclass[12pt]{amsart}
\usepackage{amssymb, amscd}
\usepackage{tikz}

\usepackage{mathptmx}

\usepackage[latin1]{inputenc}
\newcommand{\R}{\mathbb{R}}
\newcommand{\N}{\mathbb{N}}

\newcommand{\gN}{\text{g}N}
\newcommand{\gM}{\text{g}M}

\DeclareMathOperator{\pnt}{\raise 0.5mm \hbox{\large\bf.}}

\DeclareMathOperator{\conv}{conv}

\DeclareMathOperator{\GL}{GL}
\DeclareMathOperator{\inom}{in}

\DeclareMathOperator{\gin}{gin}

\DeclareMathOperator{\GF}{GF}
\DeclareMathOperator{\gGF}{gGF}

\newtheorem{theorem}{\bf Theorem} [section]
\newtheorem{lemma}[theorem]{\bf Lemma}
\newtheorem{cor}[theorem]{\bf Corollary}
\newtheorem{prop}[theorem]{\bf Proposition}

\newtheorem{quest}[theorem]{\bf Question}

\theoremstyle{definition}

\newtheorem{notation}[theorem]{\bf Notation}
\newtheorem{rem}[theorem]{\bf Remark}

\theoremstyle{plain}
\newtheorem*{satz*}{Theorem}

\title{Lower Bounds for the Number of Generic Initial Ideals}

\author{Joke Frels}
\author{Kirsten Schmitz}

\address{Joke Frels, Fachbereich Mathematik, Technische Universit\"at Kaiserslautern, Postfach 3049, 67653 Kaiserslautern, Germany}
\email{joke.frels@yahoo.de}

\address{Kirsten Schmitz, Fachbereich Mathematik, Technische Universit\"at Kaiserslautern, Postfach 3049, 67653 Kaiserslautern, Germany}
\email{schmitz@mathematik.uni-kl.de}

\thanks{Kirsten Schmitz has been supported by the DFG grant Ga 636/3}
\begin{document}

\begin{abstract}
Given a graded ideal $I$ in a polynomial ring over a field $K$ it is well known, that the number of distinct generic initial ideals of $I$ is finite. While it is known that for a given $d\in\N$ there is a global upper bound for the number of generic initial ideals of ideals generated in degree less than $d$, it is not clear how this bound has to grow with $d$. In this note we will explicitly give a family $(I(d))_{d\in\N}$ of ideals in $S=K[x,y,z]$, such that $I(d)$ is generated in degree $d$ and the number of generic initial ideals of $I(d)$ is bounded from below by a linear bound in $d$. Moreover, this bound holds for all graded ideals in $S$, which are generic in an appropriate sense.
\end{abstract}

\maketitle

\section{Introduction}

Generic initial ideals are useful tools in commutative algebra reflecting homological and algebraic properties of the original ideal in a direct way, see \cite{E,GRE}. Introduced in \cite{BAST} to study the regularity of graded ideals they also have various applications in algebraic geometry, see for example \cite{HE}. While generic initial ideals, gins for short, with respect to certain term orders (in particular, the reverse lexicographic term order) have been studied well, little is known about gins with respect to other term orders. In particular, it is a natural question of how many generic initial ideals an ideal in a fixed polynomial ring can have. For procedures such as the Gr\"obner walk for fast computations of Gr\"obner bases, see \cite{COKAMA}, it is of course useful to have information on the number of full-dimensional cones in a Gr\"obner fan. Asking for the number of generic initial ideals means studying this issue in the generic setting.

Note that it is easy to construct a family of ideals such that the number of distinct initial ideals (or equivalently the number of full-dimensional cones in the Gr\"obner fan) increases. We are, however, interested in the number of generic initial ideals of an ideal (or equivalently, maximal cones in the generic Gr\"obner fan).

Let $K$ be a field and $I$ a graded ideal in a polynomial ring in one or two variables over $K$. Then the number of generic initial ideals of $I$ can be at most two. This follows from the fact that the generic Gr\"obner fan is $\R$ in the case of one variable and either $\R^2$ or the fan in $\R^2$ consisting of the cone $\R(1,1)$ and the two maximal cones induced by it. For three or more variables, however, the number of generic initial ideals is not so easy to control. In this note we will therefore deal with the following question.

\begin{quest}\label{question}
Given a natural number $k$ is there a graded ideal in $K[x,y,z]$ with at least $k$ distinct generic initial ideals?
\end{quest}

We consider this question in a polynomial ring over a field $K$ of characteristic $0$ (the assumption on the characteristic is necessary for the proof of Theorem \ref{main}). We will give a positive answer by explicitly describing a family of monomial ideals in $K[x,y,z]$ such that for each $k\in \N$ we can point to an ideal in the family having at least $k$ distinct generic initial ideals. 

In \cite{MOMO} relying on \cite{BA} it was shown that there is a bound on the maximal degree of the elements of a universal Gr\"obner basis of $I\subset K[x,y,z]$ which is a function of $\dim(S/I)$ and the maximal degree in a given generating set of $I$. For our setting this means that the there is an upper bound for the number of generic initial ideals of an ideal generated in a given degree. So, to exhibit a family of ideals with an increasing number of gins, we certainly have to increase the degree of the generators. This is not a sufficient condition though: Computations with Singular \cite{DEGRPFSC} and Gfan \cite{JE} indicate that, for example, the number of gins of $I=(x^d,x^{d-1}y,\ldots,xy^{d-1},y^d)\subset S$ is always $3$ independently of $d$. We will show, however, that for the family $(I(d))_{d\geq3}$ with $I(d)=(x^d,x^{d-1}y,\ldots,xy^{d-1},z^d)\subset S$ there is a lower bound for the number of generic initial ideals of $I(d)$, which is linear in $d$, see Theorem \ref{main}. From this we derive that this bound also holds for a class of graded ideals satisfying a certain genericity assumption, see Theorem \ref{genericresult}.

\section{Bounds for Generic Initial Ideals}

\subsection{Genericity}\label{gammamap}

We will consider graded ideals in $S=K[x,y,z]$ with respect to the standard grading. To deal with linear coordinate changes induced by various $g\in \GL_3(K)$ simultaneously it is useful to replace $K$ by a polynomial ring $K[\Gamma]$ over $K$, where we set $\Gamma=\left\{\gamma_1,\ldots,\gamma_9\right\}$. This allows one to perform calculations in the polynomial ring over $K[\Gamma]$ and afterwards evaluate at appropriate $g\in \GL_3(K)$. In the following we will consider the $K$-algebra homomorphism induced by

\begin{eqnarray*}
\gamma: K[x,y,z] & \longrightarrow & K[\Gamma][x,y,z]\\
x        & \longmapsto     & \gamma_1x+\gamma_2y+\gamma_3z\\
y        & \longmapsto     & \gamma_4x+\gamma_5y+\gamma_6z\\
z        & \longmapsto     & \gamma_7x+\gamma_8y+\gamma_9z.
\end{eqnarray*}

Note that for an ideal $I$ the image $\gamma(I)$ is not an ideal. By abuse of notation we will, however, denote by $\gamma(I)$ the ideal generated by this image. For $g\in \GL_3(K)$ evaluating $\gamma_i$ at $g_i$ induces a linear coordinate transformation on $K[x,y,z]$, which by abuse of notation we will denote by $g$ as well. It is well known that for a given term order $\succ$ there is a Zariski-open set $\emptyset\neq U\subset \GL_3(K)$ such that $\inom_{\succ}(g(I))$ is the same ideal for all $g\in U$. This ideal $\gin_{\succ}(I)$ is the generic initial ideal of $I$ with respect to $\succ$. 

\subsection{The generic Gr\"obner fan and its graded components}

To prove the existence of a given number of distinct generic initial ideals we will use the one-to-one correspondence between gins and the maximal cones of the generic Gr\"obner fan. Let $I\subset S=K[x,y,z]$ be a graded ideal. The Gr\"obner fan $\GF(I)$ of $I$ as introduced in \cite{MORO} is the set of equivalence classes of $\omega\in\R^n$ of the relation defining $\omega,\omega'\in\R^3$ to be equivalent if $\inom_{\omega}(I)=\inom_{\omega'}(I)$. By \cite[Theorem 1.1]{TK} there exists a Zariski-open set $\emptyset\neq U\subset\GL_3(K)$ such that $\GF(g(I))$ is the same fan for all $g\in U$. This fan is called the \emph{generic Gr\"obner fan of $I$} and denoted by $\gGF(I)$. The maximal cones of $\gGF(I)$ are in one-to-one correspondence with the distinct generic initial ideals of $I$.

One description of $\GF(I)$ results from comparing initial ideals by comparing their graded components. In particular, we can consider the ``degree $d$ part'' of the Gr\"obner fan for each $d\in \N$ by looking at the equivalence relation that defines $\omega$ to be equivalent to $\omega'$ if $\inom_{\omega}(I)_d=\inom_{\omega'}(I)_d$. The same arguments showing that the Gr\"obner fan is indeed a fan can be used to prove that the set of the closures of the equivalence classes of this relation is a complete fan in $\R^3$. We will denote this fan by $\GF(I)_d$. As $\GF(I)$ is a refinement of $\GF(I)_d$ and, indeed, also $\gGF(I)$ is a refinement of $\gGF(I)_d$, the number of distinct maximal cones in $\gGF(I)_d$ provides a lower bound for the number of distinct generic initial ideals of $I$.

One advantage of studying the Gr\"obner fan of $I$ via its graded components is that the defining equations of its cones can be expressed in the Pl\"ucker coordinates of $I_d$. Recall that for a subspace $W$ of a finite dimensional $K$-vector space $V$ with a given basis $B$ of $V$ the Pl\"ucker coordinates of $W$ can be computed in the following way: Choose a basis of $W$ and express this basis in the elements of $B$ obtaining a $\dim(W)\times \dim(V)$-matrix $A$ with entries in $K$. The vector of determinants of the maximal minors of $A$ does not depend on the choice of the basis of $W$ up to nonzero scalar multiple. This vector, considered as an element in projective $(\dim(V)-1)$-space, is called \emph{the Pl\"ucker coordinates of $W$}. In our case $V$ will be $K[x,y,z]_d$ for a given $d$, and $B$ will be the basis consisting of all degree $d$-monomials indexed by their exponents. When we talk about an entry of the Pl\"ucker coordinates $P$, we will mean an entry of any representative of $P$ in homogeneous coordinates. As we will only be concerned with the question of whether an entry is zero or not, our arguments will not depend on the choice of representative.

We will use the following notation throughout this note.

\begin{notation}\label{pluck}
Fix $d\in \N$. For a monomial $x^{\nu_1}y^{\nu_2}z^{\nu_3}$ in $K[x,y,z]$ we can consider its exponent as a vector $(\nu_1,\nu_2,\nu_3)\in\N^3$. Consider all sets $J$ of exponents of degree $d$ monomials in $x,y,z$ with $|J|=d+1$ and denote by $N(d)$ the set of all such $J$. For $J\in N(d)$ let $m_J=\sum_{\nu\in J} \nu\in \N^3$ and denote by $M(d)$ the set of $m_J$ with $J\in N(d)$. For a graded ideal $I\subset S$ assume that $\dim_K(I_d)=d+1$ and let $P_J(I_d)$ be the entry of the Pl\"ucker coordinates of $I_d$ defined by $J\in N(d)$. We set $N(I,d)=\left\{J\in N(d): P_J(I_d)\neq 0\right\}$ and $$M(I,d)=\left\{m\in M(d): \exists\ J\in N(I,d): m=m_J \right\}.$$ 
\end{notation}

Each maximal cone of $\GF(I)_d$ can now be described by one element of $M(I,d)$.

\begin{prop}\label{descriptionofonegroeb}
For each maximal cone $C$ in $\GF(I)_d$ there exists a unique $m\in M(I,d)$ such that $C=\left\{\omega\in \R^3: \omega\cdot m\leq \omega\cdot m_{J} \text{ for } m_J\in M(I,d) \right\}$. The map associating to $C$ the corresponding $m$ is injective.
\end{prop}

\begin{proof}
This statement follows from the proof of the existence of the Gr\"obner complex as explained in \cite{MAST} in a Chapter on Gr\"obner basis theory (currently in the proof of Theorem 2.4.11). In this setting the field $K$ is considered together with a valuation $v:K\longrightarrow \R\cup\left\{\infty\right\}$ and initial forms and ideals are defined with respect to the valuations of the coefficients of the polynomials. The equivalence classes of $\omega\in \R^3$ of inducing the same graded component of an initial ideal are relatively open polyhedra, see \cite{MAST}. By use of Notation \ref{pluck} the defining equations of such a Gr\"obner polyhedron $C\in\GF(I)_d$ are determined by giving a subset $A$ of $N(d)$: A vector $\omega$ is contained in the relative interior of $C$ if and only if
\begin{eqnarray*}
v(P_J(I_d))+\omega\cdot m_J & = & v(P_{J'}(I_d))+\omega\cdot m_{J'} \text{ for } J,J'\in A\\
v(P_J(I_d))+\omega\cdot m_J & < & v(P_{J'}(I_d))+\omega\cdot m_{J'} \text{ for } J\in A, J'\in N(d)\backslash A.
\end{eqnarray*}
Our case (the constant coefficient case of \cite{MAST}) corresponds to considering $K$ equipped with the trivial valuation with $v(0)=\infty$ and $v(a)=0$ for all $a\neq0$. Hence, in our setting all defining equations of Gr\"obner cones are of the form
\begin{eqnarray*}
\omega\cdot m_J & = & \omega\cdot m_{J'} \text{ for } J,J'\in A^*\\
\omega\cdot m_J & < & \omega\cdot m_{J'} \text{ for } J\in A^*, J'\in N(I,d)\backslash A,
\end{eqnarray*}
where $A^*= N(I,d)\cap A$. To define a maximal cone of $\GF(I)$ the set $A$ cannot contain $J,J'$ with $m_J\neq m_{J'}$, as otherwise there would be at least one equality among the defining relations. Thus there is a unique $m\in M(I,d)$ with $m_J=m$ for $J\in A$. As $\omega,\omega'$ are in the same relatively open cone of $\GF(I)$ if and only if the same of the above equalities and inequalities are fulfilled, the assignment of $m$ to $C$ as described above is injective. This proves the claim.
\end{proof}

With the same argument we can determine the defining inequalities of maximal cones in $\gGF(I)_d$. As $M(I,d)$ depends on the ideal in question, we need to prove that $M(g(I),d)=M(g'(I),d)$ for all $g,g'$ in some nonempty Zariski-open subset of $\GL_3(K)$. To do this note that for fixed $d\in\N$ we have $\dim_K(g(I)_d)=d+1$ for every $g\in\GL_3(K)$ if $\dim_K(I)=d+1$, so we have to consider the same $N(d)$ and $M(d)$ for every ideal $g(I)$ for $g\in\GL_3(K)$. To ensure the same for $N(g(I),d)$, and thus for $M(g(I),d)$ for generic $g$ note that the Pl\"ucker coordinates of $\gamma(I)_d$ can be considered as polynomials the $\gamma_i$. As there are only finitely many $J\in N(d)$, there exists $\emptyset\neq U\subset \GL_3(K)$ and a subset  $N\subset N(d)$ (independent of $g$) such that $P_J(g(I)_d)\neq0$ for $J\in N$ and $P_J(g(I)_d)=0$ for $J\in N(d)\backslash N$ for all $g\in U$. Thus, $N(g(I),d)$ and $M(g(I),d)$ are the same sets (respectively) for all $g\in U$.

\begin{notation}\label{genN}
The set $N$ such that $P_J(g(I)_d)\neq0$ for generic $g$ as described above we will denote by $\gN(I,d)$. Analogously we set $$\gM(I,d)=\left\{m\in M(d):\exists\ J\in \gN(I,d): m=m_J  \right\}.$$
\end{notation}

We now immediately get the analogous result to Proposition \ref{descriptionofonegroeb} for the generic case.

\begin{cor}\label{descriptionofgroeb}
For each maximal cone $C$ in $\gGF(I)_d$ there exists a unique $m\in\gM(I,d)$ such that $C=\left\{\omega\in \R^3: \omega\cdot m\leq \omega\cdot m_{J} \text{ for } m_J\in \gM(I,d) \right\}$. As in Proposition \ref{descriptionofonegroeb} the map associating to $C$ the corresponding $m$ is injective
\end{cor}

\subsection{Candidates for maximal Gr\"obner cones}

Let $I\subset S=K[x,y,z]$ be a graded ideal. As the number of generic initial ideals of $I$ is equal to the number of maximal cones in the generic Gr\"obner fan $\gGF(I)$, we can express a lower bound for the number of gins in terms of the number of cones of $\gGF(I)$. Moreover, it suffices to give a lower bound for the number of maximal cones the in degree $d$ part $\gGF(I)_d$ of the generic Gr\"obner fan for some $d$, as $\gGF(I)$ is a refinement on $\gGF_d(I)$. 

Fix $d\in \N$, $d\geq 3$, and let $N(d)$ and $m_J$ for $J\in N(d)$ be as in Notation \ref{pluck}. Consider the polytope $Q(d)=\conv(m_J: J\in N(d))\subset \R^3$. This is a $2$-dimensional polytope in the plane $H=\left\{x\in \R^3: \sum_i x_i= d(d+1)\right\}$. In this section we will determine some vertices of $Q(d)$, which will correspond to maximal cones of the generic Gr\"obner fan in degree $d$ under certain circumstances.

\begin{notation}\label{indexsets}
For $0\leq n<\frac{d}{3}$ we will use the notation $J(n)$ for the set $$\left\{(d-a-1,a,1):0\leq a\leq n-1 \right\}\cup \left\{(d-b,b,0):0\leq b\leq d-n\right\}\in N(d).$$%Moreover, let $$\JJ=\left\{J(n): 0\leq n< \frac{d}{3}\right\}.$$
\end{notation}

We will now show that every such set corresponds to a vertex of $Q(d)$.

\begin{prop}\label{verticesofq}
For every $d\geq 3$ and  $0\leq n<\frac{d}{3}$ the point $m_{J(n)}$ is a vertex of $Q(d)$.
\end{prop}

\begin{proof}
A vector $m_{J}$ for some $J\in N(d)$ is a vertex of $Q(d)$ if and only if there exists $\omega\in \R^3$ such that $\omega\cdot m_J<\omega\cdot m_{J'}$ for every $J'\in N(d)$ with $m_J\neq m_{J'}$. Let $$\Delta=\left\{\nu\in\N^3:\nu_1+\nu_2+\nu_3=d\right\}.$$ To show that $m_J$ is a vertex of $Q(d)$ it thus suffices to show that there exists $\omega\in \R^3$ with $\omega\cdot \nu<\omega\cdot \nu'$ for every $\nu\in J$, $\nu'\in \Delta\backslash J$, since then we have $$\omega\cdot m_J=\omega \cdot \sum_{\nu\in J} \nu=\sum_{\nu\in J} \omega\cdot \nu<\sum_{\nu\in J'}\omega\cdot \nu=\omega\cdot m_{J'},$$ where the strict inequality is true, as there is at least one $\nu\in J'\backslash J$. Geometrically this idea can viewed as finding a line in $H\subset \R^3$ separating the points in $J$ from the ones not in $J$, see Figure \ref{fig1}.

For $0\leq n<\frac{d}{3}$ let $$\omega(n)=(2n-d-2,2n-d+1,2d-4n+1)$$ and $$\lambda(n)=d+2nd-d^2-3n.$$ By direct calculation one can show that we have $\omega(n)\cdot \nu\leq \lambda(n)$ for all $\nu\in J(n)$ and that $\omega(n)\cdot \nu> \lambda(n)$ for all $\nu\in \Delta\backslash J(n)$.  Thus $m_{J(n)}$ is a vertex of $Q(d)$ with defining hyperplane $\left\{x\in \R^3: \omega(n)\cdot x=\lambda(n)\right\}$.
\end{proof}

\begin{figure}\centering
\begin{tikzpicture}[scale=0.7]

\draw (0,0) -- (7,0) -- (60:7cm) -- cycle;
\draw[very thick] (2,0) -- (10.5:8cm)
 node[pos=0.9, above] {$l$};
\draw[very thick] (2,0) -- (42:-1cm);
\node[right] at (7,0) {$(d,0,0)$};
\node[left] at (0,0) {$(0,d,0)$};
\node[right] at ++(60:7cm) {$(0,0,d)$};
%\node at (3.5,3) {$H_-$};
%\node at (6,0.5) {$H_+$};

\fill[black] (0,0) circle [radius=0.1];
\fill[black] (1,0) circle [radius=0.1];
\fill[black] (2,0) circle [radius=0.1];
\fill[black] (3,0) circle [radius=0.1];
\fill[black] (4,0) circle [radius=0.1];
\fill[black] (5,0) circle [radius=0.1];
\fill[black] (6,0) circle [radius=0.1];
\fill[black] (7,0) circle [radius=0.1];

\fill[black] (0,0)++(60:1cm) circle [radius=0.1];
\fill[black] (0,0)++(60:2cm) circle [radius=0.1];
\fill[black] (0,0)++(60:3cm) circle [radius=0.1];
\fill[black] (0,0)++(60:4cm) circle [radius=0.1];
\fill[black] (0,0)++(60:5cm) circle [radius=0.1];
\fill[black] (0,0)++(60:6cm) circle [radius=0.1];
\fill[black] (0,0)++(60:7cm) circle [radius=0.1];

\fill[black] (1,0)++(60:1cm) circle [radius=0.1];
\fill[black] (1,0)++(60:2cm) circle [radius=0.1];
\fill[black] (1,0)++(60:3cm) circle [radius=0.1];
\fill[black] (1,0)++(60:4cm) circle [radius=0.1];
\fill[black] (1,0)++(60:5cm) circle [radius=0.1];
\fill[black] (1,0)++(60:6cm) circle [radius=0.1];

\fill[black] (2,0)++(60:1cm) circle [radius=0.1];
\fill[black] (2,0)++(60:2cm) circle [radius=0.1];
\fill[black] (2,0)++(60:3cm) circle [radius=0.1];
\fill[black] (2,0)++(60:4cm) circle [radius=0.1];
\fill[black] (2,0)++(60:5cm) circle [radius=0.1];

\fill[black] (3,0)++(60:1cm) circle [radius=0.1];
\fill[black] (3,0)++(60:2cm) circle [radius=0.1];
\fill[black] (3,0)++(60:3cm) circle [radius=0.1];
\fill[black] (3,0)++(60:4cm) circle [radius=0.1];

\fill[black] (4,0)++(60:1cm) circle [radius=0.1];
\fill[black] (4,0)++(60:2cm) circle [radius=0.1];
\fill[black] (4,0)++(60:3cm) circle [radius=0.1];

\fill[black] (5,0)++(60:1cm) circle [radius=0.1];
\fill[black] (5,0)++(60:2cm) circle [radius=0.1];

\fill[black] (6,0)++(60:1cm) circle [radius=0.1];

\end{tikzpicture}
\caption{The points in $\Delta$ for $d=7$ with the line $l$ separating the points of $J(2)$ from the others.}\label{fig1}

\end{figure}
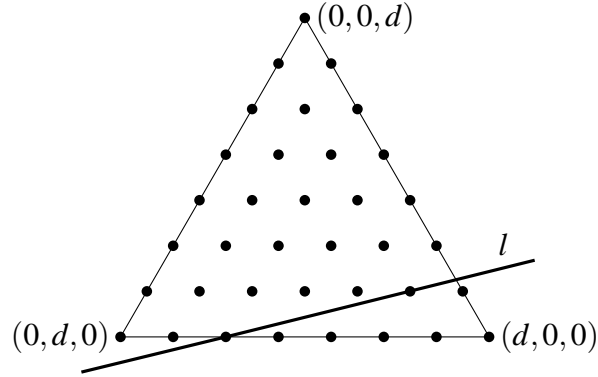

\begin{cor}\label{mehrecken}
For $d\in \N$ the polytope $Q(d)$ has at least $\frac{d}{3}$ vertices. %In particular, for every $m\in \N$, there exists $d_0\in \N$ such that $Q(d)$ has at least $m$ vertices for all $d\geq d_0$.
\end{cor}

\begin{proof}
As the last coordinate of $m_{J(n)}$ is $n$, the $m_{J(n)}$ are distinct for distinct $n$. By Proposition \ref{verticesofq} each $m_{J(n)}$ for $0\leq n<\frac{d}{3}$ is a vertex of $Q(d)$.
\end{proof}

\subsection{Main Result}\label{mainsec}

Let $d\in\N$. The aim of this section is to show that for almost all graded ideals in $S=K[x,y,z]$ generated by $d+1$ homogeneous polynomials of degree $d$ the number of generic initial ideals is bounded from below by $\frac{d}{3}$, see Theorem \ref{genericresult} for the precise statement. To parametrize these ideals we want to consider $d+1$ polynomials of degree $d$ whose coefficients can be interpreted as variables, which can then be substituted by elements of $K$. More precisely, we will use the following notation.

\begin{notation}\label{extnot}
Let $\Delta=\left\{\nu\in \N^3: \nu_1+\nu_2+\nu_3=d\right\}$ and consider the polynomial ring $$L=K[b_{i\nu}: 1\leq i\leq d+1, \nu\in \Delta]$$ over $K$. Set $$f_i=\sum_{\nu\in \Delta} b_{i\nu}x^{\nu}\in L[x,y,z]$$ for $1\leq i \leq d+1$. For $q=(q_{i\nu})_{i\nu}\in K^{(d+1)|\Delta|}$ by abuse of notation we will denote $\sum_{\nu\in \Delta} q_{i\nu}x^{\nu}\in S$ by $f_i(q)$ and the ideal generated by $f_i(q)$ for $i=1,\ldots,d+1$ by $I(d,q)$.
\end{notation}

In this way the affine space $K^{(d+1)|\Delta|}$ parametrizes graded ideals generated by $d+1$ polynomials of degree $d$. This assignment of points in $K^{(d+1)|\Delta|}$ to graded ideals is of course not injective, but we do not need it to be for the following.

To start we will give a sufficient condition for $I(d,q)$ to have at least $\frac{d}{3}$ distinct generic initial ideals in terms of certain Pl\"ucker coordinates not being zero. We will then proceed by exhibiting a family of ideals that fulfill these conditions, i.e. for every $d$ we will obtain an explicit $q\in K^{(d+1)|\Delta|}$ such that Lemma \ref{nonzerojn} can be applied to $I(d,q)$. Finally we can show that this result implies that for each $d$ every ideal $I(d,q)$ fulfills these conditions for generic enough $q$.

\begin{lemma}\label{nonzerojn}
Let $d\in\N$ with $d\geq 3$, $q\in K^{(d+1)|\Delta|}$ and $I(d,q)\subset S$ as defined above. If $\dim_K(I(d,q)_d)=d+1$ and if $J(n)\in \gN(I(d,q),d)$ as defined in Notation \ref{genN} for every $J(n)$ as defined in Notation \ref{indexsets}, then $I(d,q)$ has at least $\frac{d}{3}$ distinct generic initial ideals.
\end{lemma}

%For simplicity denote $\gN(I(d,q),d)$

\begin{proof}
We will prove that $\frac{d}{3}$ is a lower bound for the number of full-dimensional cones in the graded component $\gGF(I(d,q))_d$ of the generic Gr\"obner fan of $I(d,q)$. As the fan $\gGF(I(d,q))$ is a refinement of $\gGF(I(d,q))_d$, this gives a lower bound on the number of full-dimensio\-nal cones in the generic Gr\"obner fan of $I(d,q)$ and, thus, for the number of generic initial ideals of $I(d,q)$.

By the proof of Proposition \ref{verticesofq} we know that for the sets $J(n)$ for $0\leq n< \frac{d}{3}$ and $$\omega(n)=(2n-d-2,2n-d+1,2d-4n+1)$$ we have $\omega(n)\cdot m_{J(n)}< \omega(n)\cdot m_{J'}$ for all $J(n)\neq J'\in N(d)$. Moreover, $m_{J(n)}\neq m_{J(n')}$ for $n\neq n'$. By assumption $J(n)\in \gN(I(d,q),d)$, so $m_{J(n)}\in \gM(I(d,q),d)$ for every $n$. Thus, by Corollary \ref{descriptionofgroeb} the $\omega(n)$ for $0\leq n< \frac{d}{3}$ are all contained in different maximal cones of $\gGF(I(d,q))_d$.
\end{proof}

We will now give a family $(I(d))_{d\in \N}$ of ideals in $K[x,y,z]$ such that $I(d)$ fulfills the conditions from Lemma \ref{nonzerojn} and thus the family is an example class providing a positive answer to Question \ref{question}.

\begin{notation}\label{family}
Consider the family of ideals $(I(d))_{d\geq 3}$ such that 
\begin{eqnarray*}
I(d) & = & (x^ay^{d-a},z^d: 0\leq a\leq d-1)\\
     & = & (y^d,xy^{d-1},x^2y^{d-2},\ldots,x^{d-1}y,z^d).
\end{eqnarray*}
Note that $I(d)$ is generated in degree $d$ by $d+1$ monomials.
\end{notation}

\begin{rem}\label{fam}
Note that $I(d)$ is equal to $I(d,q)$ for $q\in K^{(d+1)|\Delta|}$ defined as follows: Let $\nu_i=(i-1,d-i+1,0)$ for $i=1,\ldots,d$ and $\nu_{d+1}=(0,0,d)$. Consider the evaluation map $\phi:L\longrightarrow K$ mapping $b_{i\nu_i}$ to $1$ for $i=1,\ldots,d+1$ and all other independent variables to $0$. Then $I(d)=I(d,q)$ for $(q_{i\nu})_{i\nu}=(\phi(b_{i\nu}))_{i\nu}$.
\end{rem}

\begin{theorem}\label{main}
The ideal $I(d)$ as defined in Notation \ref{family} has at least $\frac{d}{3}$ distinct generic initial ideals.
\end{theorem}

\begin{proof}
Let $d$ be fixed and for simplicity denote $\gN(I(d),d)$ by $N$. As $I(d)$ is one of the $I(d,q)$ as defined above by Remark \ref{fam} and $\dim_K(I(d)_d)=d+1$, we can apply Lemma \ref{nonzerojn} to $I(d)$ if $J(n)\in N$ for every $n$ with $0\leq n< \frac{d}{3}$. We thus have to show that there exists $\emptyset\neq U\subset \GL_3(K)$ such that the Pl\"ucker coordinates $P_{J(n)}(g(I_d))\neq 0$ for all $g\in U$. %For this it is enough to show that the Pl\"ucker coordinates $P_J(\gamma(I(d))_d)$ are nonzero for every set $J=J(n)$. 
Choose the system of polynomials $$\gamma(y^d),\gamma(xy^{d-1}),\gamma(x^2y^{d-2}),\ldots,\gamma(x^{d-1}y),\gamma(z^d)$$ as a $K(\Gamma)$-basis of the degree $d$ component of $\gamma(I(d))_d$, where $\gamma$ is defined as in Subsection \ref{gammamap}. Fix $0\leq n< \frac{d}{3}$ and choose the ordering $$x^d,x^{d-1}y,x^{d-2}y^2,\ldots,x^ny^{d-n},x^{d-1}z,x^{d-2}yz,x^{d-3}y^2z,\ldots,x^{d-n}y^{n-1}z$$ of the monomials of degree $d$ indexed by $J(n)$. Let $B$ be the $(d+1)\times (d+1)$-matrix with entries $B_{ij}$ the coefficients of the $j$th monomial in the above ordering of monomials given by $J(n)$ in the $i$th polynomial in the above system of generators of $\gamma(I(d))_d$. Note that $B$ is a matrix over $K(\Gamma)$. It now suffices to show that $\det(B)$ is not equal to zero.
%=P_{J(n)}(\gamma(I(d))_d)

As all entries of $B$ are by definition in $K[\Gamma]$, we can evaluate each entry by setting $\gamma_i=a_i$ for some $a_i\in K$. To show that $\det(B)\neq 0$ it is enough to show that it is non-zero after an evaluation at some $a_i\in K$. We choose $\gamma_1=\gamma_3=\gamma_4=\gamma_5=\gamma_7=1$ and $\gamma_2=\gamma_6=\gamma_8=\gamma_9=0$. After this evaluation the matrix $B$ is of the blockform 

$$
\left(
\begin{array}{c|c}
B' & B''\\
\hline
b & 0 
\end{array}
\right),
$$ 

with the following submatrices:

\begin{enumerate}
\item $B'$ is an $d\times (d-n+1)$-matrix with $B'_{ij}={{d-i+1}\choose{j-1}}$ for $i=1,\ldots,d$, $j=1,\ldots,d-n+1$.
\item $B''$ is an $d\times n$-matrix with $B''_{ij}=(i-1){{d-i+1}\choose{j-1}}$ for $i=1,\ldots,d$, $j=1,\ldots,n$.
\item $b$ is a $1\times(d-n+1)$-matrix with $b_{11}=1$ and $b_{1j}=0$ for $j=2,\dots,d-n+1$.
\end{enumerate}

In this description we assume that the characteristic of $K$ is $0$. By Proposition \ref{nonzerodeterminant} in the Appendix we know that $\det(B)\neq0$ for all choices of $d$ and $n$.
\end{proof}

The fact that the ideal $I(d)$ has at least $\frac{d}{3}$ distinct generic initial ideals for $d\in \N$ can be used to show that having at least $\frac{d}{3}$ distinct gins is a generic property in the following sense.

\begin{theorem}\label{genericresult}
Let $d\in \N$, $d\geq 3$. There is a Zariski-open set $\emptyset\neq U\subset K^{(d+1)|\Delta|}$ such that $I(d,q)\subset S$ has at least $\frac{d}{3}$ distinct generic initial ideals for every $q\in U$.
\end{theorem}

\begin{proof}
By Lemma \ref{nonzerojn} we have to determine an open subset $\emptyset\neq U\subset K^{(d+1)|\Delta|}$ such that $\dim_K(I(d,q)_d)=d+1$ and $J(n)\in \gN(I(d,q),d)$ for every $J(n)$ as in Notation \ref{indexsets} and every $q\in U$. For the first condition note that there is a non-empty open subset $\tilde{U}$ of $K^{(d+1)|\Delta|}$ such that $f_1(q),\ldots,f_{d+1}(q)$ are linearly independent for $q\in \tilde{U}$. We can thus assume that $\dim_K(I(d,q)_d)=d+1$ generically. It remains to show that for generic $g\in\GL_3(K)$ the Pl\"ucker coordinates of $g(I(d,q))_d$ corresponding to the columns indexed by $J(n)$ are not equal to zero. 

Let $\gamma: L[x,y,z]\longrightarrow L[\Gamma][x,y,z]$ as in Subsection \ref{gammamap} with $L$ as defined in Notation \ref{extnot}. Let $B$ be the $(d+1)\times {{d+2}\choose{2}}$-matrix of the coefficients of $\gamma(f_1),\ldots,\gamma(f_{d+1})$, i.e. the entry $b_{ij}$ of this matrix is the coefficient of $\gamma(f_i)$ in the basis of all monomials of degree $d$ in $x,y,z$ in reverse lexicographic order. Note that these coefficients are polynomial expressions in the $\gamma_j$ and the $b_{i\nu}$ for $j=1,\ldots,9$, $i=1,\ldots,d+1$, $\nu\in\Delta$. For $(p,q)\in K^{9}\times K^{(d+1)|\Delta|}$ let $B(p,q)$ denote the $(d+1)\times {{d+2}\choose{2}}$-matrix over $K$ obtained by mapping $\gamma_j$ to $p_j$ and $b_{i\nu}$ to $q_{i\nu}$ for every $j,i,\nu$.

For $N(d)$ as in Notation \ref{pluck} and $J\in N(d)$ denote by $B_J$ the matrix consisting of all columns from $B$ indexed by the elements of $J$. By the choice of $B$ the determinant $\det(B_J)$ is a polynomial in the $\gamma_j$ and $b_{i\nu}$ with coefficients in $K$. For $(p,q)\in K^{(d+1)|\Delta|}$ we have $\det(B_J)(p,q)=\det(B(p,q))$.

With the notation of Remark \ref{fam} for $(p,q)\in K^9\times K^{(d+1)|\Delta|}$ with $q_{i\nu_i}=1$ for $i=1,\ldots,d+1$ and $0$ otherwise, and $p_1=p_3=p_4=p_5=p_7=1$ and $p_2=p_6=p_8=p_9=0$ we have that $B(p,q)$ is exactly the matrix with rows $$\gamma(y^d),\gamma(xy^{d-1}),\gamma(x^2y^{d-2}),\ldots,\gamma(x^{d-1}y),\gamma(z^d).$$ By Proposition \ref{nonzerodeterminant} we know that $\det(B(p,q)_{J(n)})\neq 0$ for every $J(n)$ from Notation \ref{indexsets}. Hence, $\det(B_{J(n)})$ is not the zero polynomial in $K[\gamma_1,\ldots,\gamma_9][b_{i\nu}:1\leq i\leq d+1, \nu\in \Delta]$. Let $\emptyset\neq V\subset K^{9+(d+1)|\Delta|}$ be an open subset such that $\det(B(p,q)_{J(n)})\neq 0$ for every $(p,q)\in V$ and every $J(n)$. 

Let $$U=\left\{q\in K^{(d+1)|\Delta|}: \text{ there exists } p\in K^9: (p,q)\in V  \right\}\subset K^{(d+1)|\Delta|},$$ which is a non-empty open subset of $K^{(d+1)|\Delta|}$. For each $q\in U$ if we substitute the $b_{i\nu}$ by the corresponding $q_{i\nu}$ in $\det(B_{J(n)})$, we obtain a polynomial in $K[\gamma_1,\ldots,\gamma_9]$, which is not the zero-polynomial, since by assumption there exists $p\in K^9$ with $\det(B(p,q)_{J(n)})\neq0$. Thus for a given $q\in U$ there exists $\emptyset\neq W(q)\subset \GL_3(K)$ with $\det(B(p,q)_{J(n)})\neq0$ for every $p\in W(q)$. In other words $J(n)\in \gN(d)$ as in Notation \ref{genN} for every ideal $I(d,q)\subset S$, where $q\in U$. By Lemma \ref{nonzerojn} this implies that $I(d,q)$ has at least $\frac{d}{3}$ generic initial ideals for every $q\in U$.
\end{proof}

\section{Appendix}

This appendix contains the proof that the determinants of the matrices describing the relevant Pl\"ucker coordinates needed in Section \ref{mainsec} are not equal to zero. For $d\in \N$ and $0\leq n<d$ consider the matrix
$$
B=\left(
\begin{array}{c|c}
B' & B''\\
\hline
b & 0 
\end{array}
\right),
$$ 

with the following submatrices:

\begin{enumerate}
\item $B'$ is an $d\times (d-n+1)$-matrix with $B'_{ij}={{d-i+1}\choose{j-1}}$ for $i=1,\ldots,d$, $j=1,\ldots,d-n+1$.
\item $B''$ is an $d\times n$-matrix with $B''_{ij}=(i-1){{d-i+1}\choose{j-1}}$ for $i=1,\ldots,d$, $j=1,\ldots,n$.
\item $b$ is a $1\times(d-n+1)$-matrix with $b_{11}=1$ and $b_{1j}=0$ for $j=2,\dots,d-n+1$.
\end{enumerate}

\begin{prop}\label{nonzerodeterminant}
For every $d\in \N$ and $0\leq n<d$ we have $\det(B)\neq0$.
\end{prop}

\begin{proof}
To show that $\det(B)\neq 0$, we first do a Laplace expansion with the last row, thereby dropping the matrices $b$ and $0$ in the block form and deleting the first column of $B'$. We then replace $B''_{ij}$ by $(d+1-j)B'_{ij}-B''_{ij}$ for $j=1,\ldots,n$, which corresponds to an elementary column operation and, hence, does not change the absolute value of the determinant. This yields an $d\times d$-matrix

$$
\left(
\begin{array}{c|c}
C & C'
\end{array}
\right),
$$ 

with the submatrices $C,C'$:

\begin{enumerate}
\item $C$ is an $d\times (d-n)$-matrix with entries $C_{ij}={{d-i+1}\choose{j}}$.
\item $C'$ is an $d\times n$-matrix with entries $C'_{ij}=(d-i-j+2){{d-i+1}\choose{j-1}}$.
\end{enumerate}

We can substitute $C_{ij}$ by $C''_{ij}:=j C_{ij}$ for $j=1,\ldots,d-n$ without changing whether the determinant is zero or not. Moreover, we can replace $C''_{ij}$ by $\frac{1}{d-i+1}C''_{ij}$ and $C'_{ij}$ by $\frac{1}{d-i+1}C'_{ij}$ (i.e. multiply the $i$th row of the matrix $(C''|C')$ by $(d-i+1)$). We then obtain a matrix

$$
\left(
\begin{array}{c|c}
D & D'
\end{array}
\right),
$$ 

with 
\begin{enumerate}
\item $D$ is an $d\times (d-n)$-matrix with entries $D_{ij}=\frac{j}{d-i+1}{{d-i+1}\choose{j}}={{d-i}\choose{j-1}}$.
\item $D'$ is an $d\times n$-matrix with entries $D'_{ij}=\frac{d-i-j+2}{d-i+1}{{d-i+1}\choose{j-1}}={{d-i}\choose{j-1}}$.
\end{enumerate}

We will now inductively use row operations and Laplace expansion to eliminate the matrix $D'$ and the last $n$ rows of $(D|D')$. We replace $D_{ij}$ by $$D_{ij}-D_{(i+1)j}={{d-i}\choose{j-1}}-{{d-i-1}\choose{j-1}}={{d-i-1}\choose{j-2}}$$ and $D'_{ij}$ by $$D'_{ij}-D'_{(i+1)j}={{d-i}\choose{j-1}}-{{d-i-1}\choose{j-1}}={{d-i-1}\choose{j-2}}$$ for $i=1,\ldots,d-1$, which does not change the absolute value of the determinant. But then the first column of $D'$ is $0$ except for $D'_{d1}=1$. Using Laplace expansion on this column we get the reduced $(d-1)\times (d-1)$-matrix 

$$
\left(
\begin{array}{c|c}
D_1 & D'_1
\end{array}
\right),
$$ 

with $\det(D|D')=\det(D_1|D'_1)$, where
\begin{enumerate}
\item $D_1$ is an $(d-1)\times (d-n)$-matrix with entries $(D_1)_{ij}={{d-i-1}\choose{j-2}}$.
\item $D'_1$ is an $(d-1)\times (n-1)$-matrix with entries $(D'_1)_{ij}={{d-i-1}\choose{j-1}}$.
\end{enumerate}

This process is repeated $n$-times, so we obtain a $(d-n)\times (d-n)$-matrix $D_n$ with $(D_n)_{ij}={{d-i-n}\choose{j-n-1}}$ for $i,j=1,\ldots, d-n$, such that $\det(D_n)=\det(D|D')$.

We replace $(D_n)_{ij}$ by the entry $E_{ij}:=\frac{(d-i)!}{(d-n-i)!}\cdot\frac{(n-1+j)!}{(j-1)!}\cdot(D_n)_{ij}={{d-i}\choose{j-1}}$. As the first factor is a multiplication of each row of $D_n$ with a non-zero number and the second one is a multiplication of each column by a non-zero number, we know that $\det(D_n)\neq0$ if and only if $\det(E)\neq0$.

But $|\det(E)|=1$, which follows by induction on $d-n$. For $d-n=1$, we have the single entry ${{d-1}\choose{0}}=1$, so determinant of $U$ is $1$. Let $d-n>1$. Set $E'_{ij}=E_{(i-1)j}-E_{ij}$ for $i=2,\ldots,d-n$, which corresponds to subtracting the $i$th row of $E$ from the $(i-1)$st. Then $|\det(E)|=|\det(E')|$. As $E'_{11}=1$ and $E'_{i1}=0$ for $i>1$, we have $|\det(E')|=1\cdot \det(E'')$, where $E''_{ij}={{d-i-1}\choose{j-1}}-{{d-i}\choose{j-1}}={{d-i-1}\choose{j-2}}$ for $i,j=2,\ldots,d-n$. By the inductive hypothesis $|\det(E'')|=1$ proving the claim.
\end{proof}

\end{document}